\documentclass{article}
\usepackage{graphicx} 
\usepackage{amsthm} 
\usepackage{amssymb}
\usepackage{hyperref}
\newtheorem{theorem}{Theorem}
\newtheorem{lemma}{Lemma}
\newtheorem*{openproblem}{Open Problem}
\newtheorem{definition}{Definition}
\newtheorem*{conjecture}{Conjecture}
\usepackage{amsmath}
\usepackage{float}
\usepackage{subfigure}
\usepackage{tikz}
\usepackage{titlesec}

\newtheorem{corollary}{Corollary}

\usepackage{amsmath}
\usepackage{float}
\usepackage{subfigure}
\usepackage{tikz}
\usepackage{titlesec}
\titlelabel{\thetitle.\quad}
\usepackage{comment}

\titlelabel{\thetitle.\quad}
\title{Exploring the 3-Token Graph of Particular Graphs}
\author{\small{Felicia Servina Djuang, Arizka Yuliana,}\\\small{Widi Bagaskara, and Yeni Susanti}\\\small{Department of Mathematics Universitas Gadjah Mada, Yogyakarta Indonesia, 55281}}
\date{}

\begin{document}
\maketitle
\begin{abstract}
This study investigates the properties of the 3-token graph derived from path graphs, with a particular focus on its structural characteristics and key attributes. We analyze how the 3-token graph is constructed from path graphs and explore fundamental properties such as connectivity, diameter, and chromatic number. Furthermore, we extend our analysis to the 3-token graph of the disjoint union of two given graphs, examining its unique features and how the structure of the original graphs influences the resulting 3-token graph. The findings of this study contribute to a deeper understanding of token graphs and their applications in graph theory.
\end{abstract}

\section{Introduction}
    
    Token graphs were introduced by \cite{FabilaMonroy2012} et al. as a framework to model the movement of objects or entities across graph vertices following specific rules, such as in graph pebbling. These graphs have since garnered significant attention due to their ability to preserve and analyze structural properties.  

Given a graph \( G = (V(G), E(G)) \) with \( |V(G)| = n \), let \( k \leq n \) be a natural number and \( \Gamma_k(V(G)) \) denote the collection of all \( k \)-subsets of \( V(G) \). The \( k \)-token graph \( \Gamma_k(G) \) has \( \Gamma_k(V(G)) \) as its vertex set. Two distinct vertices \( A, B \in \Gamma_k(V(G)) \) are adjacent if and only if \( A \triangle B = \{x, y\} \), where \( xy \in E(G) \) (i.e., their symmetric difference is an edge in \( G \)).  

Recent studies have concentrated on the properties of token graphs. In 2023, Ju Zhang et al. \cite{Zhangju2023} explored the automorphisms in 2-token graphs, while in 2022, \cite{Zhangju2022} examined their edge-transitive properties. Earlier, in 2020, \cite{Deepalakshmi2020} investigated the structural characteristics of 2-token graphs. More recently, in 2023, \cite{Susanti2023} analyzed 2-token graphs in the context of disjoint union graphs, contributing to a deeper understanding of this evolving field.

While 2-token graphs are well-studied, 3-token graphs remain relatively unexplored, offering opportunities to address significant gaps in graph theory. This research investigates the properties of 3-token graphs derived from path graphs and disjoint unions. By extending the foundational work on 2-token graphs, this study introduces new frameworks to analyze vertex degree patterns, structural relationships, and graph invariants such as chromatic number, clique number, and independence number.  

Moreover, this research demonstrates isomorphisms between 3-token graphs and specific graph classes like the cubical staircase graph, uncovering structural relationships across graph families. These insights have practical implications for network flow analysis, resource allocation, and dynamic system modeling.  

By solving open problems and proposing new conjectures, this study advances the theoretical understanding of token graphs and lays a foundation for exploring $k$-token graphs of cycle graphs, and more complex graphs. This contribution is vital for applying graph theory to real-world challenges.  

All graphs in this paper are considerd to be finite, simple and undirected. We refer the reader to \cite{wilson} for general background on graph theory and for all undefined notions used in the text. For any graph $G$, the vertex set and the edge set will be simultaneously symbolized by $V(G)$ and $E(G)$. Recall that two graphs \( G \) and \( H \) are isomorphic if there exists a bijective function \( \phi : V(G) \to V(H) \) such that \( xy \in E(G) \) if and only if \( \phi(x)\phi(y) \in E(H) \). This function \( \phi \) is called an isomorphism, and the graphs are denoted as \( G \cong H \).   In graph theory, several fundamental parameters are used to describe a graph's structure. The degree of a vertex \(x \in V(G)\), denoted by \(\deg_G(x)\), is the number of vertices adjacent to \(x\). The distance between two vertices \(u\) and \(v\), denoted by \(d_G(u, v)\), represents the length of the shortest path connecting them. The diameter of a graph, \(\text{diam}(G)\), is defined as the longest shortest path in the graph, expressed as \(\text{diam}(G) = \max \{d_G(u, v) | u, v \in V(G)\}\). The clique number, \(\omega(G)\), is the size of the largest complete subgraph, while the independence number, \(\alpha(G)\), measures the maximum size of a set of vertices where no two are adjacent. The chromatic number, \(\chi(G)\), is the minimum number of colors needed to color the graph such that no two adjacent vertices share the same color. The dominating number, \(\sigma(G)\), is the minimum size of a vertex set \(S\) such that every vertex not in \(S\) is adjacent to at least one vertex in \(S\). The edge independence number, \(\alpha'(G)\), is the maximum size of a matching, where no two edges share a common vertex. These parameters collectively provide a comprehensive understanding of a graph’s structural properties. These parameters offer a comprehensive description of a graph’s structural properties.

\section{Results}
The degrees of vertices in \(k\)-token graphs are not yet fully understood in general, leaving this as an open area of research. In \cite{Deepalakshmi2020}, the authors introduced a vertex degree theorem specifically addressing the \(2\)-token graphs of a given graph, providing valuable insights into their structure. Expanding upon this foundational work, we aim to further explore the properties of token graphs by proposing a vertex degree theorem tailored to \(3\)-token graphs. This extension seeks to offer a deeper understanding of the degree patterns in these more complex graphs and contribute to the broader study of token graph theory.

We  provide a detailed description of the \(3\)-token graph constructed from two copies of a path graph. From this, we will gain insights into the structure of the \(3\)-token graph of the disjoint union of two arbitrary graphs.
\begin{lemma}\label{Lemma1}
    Given two path graphs $P_n^1$ and $P_n^2$, with $n \geq 3$. Let $2P_n=P_n^1 \oplus P_n^2$. It follows that
        $$\Gamma_3(2P_n)=2\Gamma_3(P_n) \oplus 2(\Gamma_2(P_n) \square P_n).$$
\end{lemma}
\begin{proof}
    According to the definition of the 3-token graph, we have
    $$V(\Gamma_3(2P_n))=\{ A \subseteq (V(P_n^1) \cup V(P_n^2)): |A|=3 \}.$$
    Let $V(P_n^1)=\{ a_i : i=1,2,\dots,n \}$ and $V(P_n^2)= \{ b_i : i=1,2,\dots,n \}$. We make a partition $\{ W_1,W_2,W_3,W_4 \}$ on $V(P_n^{1}) \cup V(P_n^2)$, where
    $$W_i=\{A \subseteq V(\Gamma_3(2P_n)): A \subseteq V(P_n^i)) \text{, for } i=1,2,$$
    $$W_3=\{A \subseteq V(\Gamma_3(2P_n)): |A \cap V(P_n^1)|=1 \text{ and } |A \cap V(P_n^2)|=2 \},$$
    $$W_4=\{A \subseteq V(\Gamma_3(2P_n)): |A \cap V(P_n^1)|=2 \text{ and } |A \cap V(P_n^2)|=1\}.$$

    It is clear that $\Gamma_3(P_n^1)$ and $\Gamma_3(P_n^2)$ are subgraphs of $\Gamma_3(2P_n)$, so that $2\Gamma_3(P_n)$ is subgraph of $\Gamma_3(2P_n)$. Now we will show that each two vertices in $\Gamma_3(P_n^1)$ or in $\Gamma_3(P_n^2)$ are not connected by any single vertex in $\Gamma_3(2P_n)$. We know that $V(\Gamma_3(P_n^1))=W_1$ and $V(\Gamma_3(P_n^2))=W_2$. Clearly, for any $A \in W_1$ and for any $B \in W_2$, it follows that $A \cap B = \emptyset$. Therefore $|A \triangle B| = 6$, implying $AB \notin E(\Gamma_3(2P_n))$. For any $A \in W_1$ and for any $C \in W_3 \cup W_4$, we consider the following cases.
    \begin{itemize}
        \item[(i)] Case $C \in W_3$. \\
        Let $C=\{a_i,b_j,b_k\}$ with $i,j,k \in \{1,2,\dots,n\}$, we obtain $|A \cap C| \leq 1 $. Therefore $|A \triangle C| \geq 4$, and thus $AC \notin E(\Gamma_3(2P_n)).$
        \item[(ii)] Case $C \in W_4$.\\ 
        Let $C=\{a_i, a_j, b_k\}$ with $i,j,k \in \{1,2,\dots,n\}$. If $\{ a_i,a_j\} \nsubseteq A$, then $AC \notin E(\Gamma_3(2P_n))$. If $\{ a_i,a_j\} \subseteq A$, then we have $A \triangle C = \{a_l,b_k\}$ with $l \in \{1,2,\dots,n\}$. As $a_lb_k \notin (E(P_n^1) \cup E(P_n^2))=E(2P_n)$, we get $AC \notin E(\Gamma_3(2P_n))$.
    \end{itemize}
    For any $B \in W_2$ and for any $C \in W_3 \cup W_4$, we get $BC \notin E(\Gamma_3(2P_n))$.

    Now, let $C=\{ a_i,b_j,b_k \} \in W_3$ and $D=\{ a_m, a_n, b_p \} \in W_4$ with $i,j,k,m,n,p \in \{1,2,\dots,n\}$.
    \begin{itemize}
        \item[(i)] If ($i=m$ or $i=n$) and ($p=j$ or $p=k$), we get $C \triangle D =\{a_t,b_u\}$ where $t \in \{m,n\}$ and $u \in \{j,k\}$. Hence $CD \notin E(\Gamma_3(2P_n)\}$.
        \item[(ii)] If ($i \neq m$ and $i \neq n$) or ($p \neq j$ and $p \neq k$), obviously $|C \triangle D| \geq 4$, so that $CD \notin E(\Gamma_3(2P_n))$.
    \end{itemize}

    Now, consider the subgraphs of $\Gamma_3(2P_n)$ induced by $W_3$ and $W_4$. Firstly, let $A,B \in W_3$ with $A \neq B$. Let $A=\{a_i,b_j,b_k\}$ and $B=\{a_m, b_n, b_p\}$ with $i,j,k,m,n,p \in \{1,2,\dots,n\}$. Clearly, $(i,j,k) \neq (m,n,p)$ and $(i,j,k) \neq (m,p,n)$. We will consider some cases:
    \begin{itemize}
        \item [1. ] If $i=m$ and $j=n$ and $k \neq p$, then $A \triangle B = \{b_k,b_p\}$. Thus, $AB \in E(\Gamma_3(2P_n))$ if and only if $\{b_j,b_k\}\{b_n,b_p\} \in E(\Gamma_2(P_n^2))$.
        \item [2. ] If $i=m$ and $j=p$ and $k \neq n$, then $A \triangle B = \{b_k,b_n\}$. Thus, $AB \in E(\Gamma_3(2P_n))$ if and only if $\{b_j,b_k\}\{b_n,b_p\} \in E(\Gamma_2(P_n^2))$.
        \item [3. ] If $i=m$ and $k=n$ and $j \neq p$, then $A \triangle B = \{b_j,b_p\}$. Thus, $AB \in E(\Gamma_3(2P_n))$ if and only if $\{b_j,b_k\}\{b_n,b_p\} \in E(\Gamma_2(P_n^2))$.
        \item [4. ] If $i=m$ and $k=p$ and $j \neq n$, then $A \triangle B = \{b_j,b_n\}$. Thus, $AB \in E(\Gamma_3(2P_n))$ if and only if $\{b_j,b_k\}\{b_n,b_p\} \in E(\Gamma_2(P_n^2))$.
        \item [5. ] If $i=m$ and $j,k \neq n$ and $j,k\neq p$, then clearly $|A \triangle B| = 4$, so $AB \notin E(\Gamma_3(2P_n))$
        \item [6. ] If $i \neq m$ and $\{j,k\}=\{n,p\}$, then $A \triangle B = \{a_i,a_m\}$. Thus, $AB \in E(\Gamma_3(2P_n))$ if and only if $a_ia_m \in E(P_n^1)$.
        \item [7. ] If $i \neq m$ and $j \in \{n,p\}$ and $k \notin \{n,p\}$, then clearly $|A \triangle B| = 4$, so that $AB \notin E(\Gamma_3(2P_n))$
        \item [8. ] If $i \neq m$ and $k \in \{n,p\}$ and $j \notin \{n,p\}$, then clearly $|A \triangle B| = 4$, so $AB \notin E(\Gamma_3(2P_n))$
        \item [9. ] If $i \neq m$ and $j,k \neq n$ and $j,k\neq p$, then clearly $|A \triangle B| = 6$, so $AB \notin E(\Gamma_3(2P_n))$.
    \end{itemize}
From these cases, for any $i,j,k,m,n,p \in \{1,2, \dots, n\}$ we have that  
$$\{a_i,b_j,b_k\}\{a_m,b_n,b_p\} \in E(\Gamma_3(2P_n))$$ if and only if either $i=m$ which implies $\{b_j,b_k\}\{b_n,b_p\} \in E(\Gamma_2(P_n^2))$ or $\{j,k\}=\{n,p\}$ which implies $a_ia_m \in E(P_n^1)$.

Now, construct a mapping $f: W_3 \to V(\Gamma_2(P_n)) \times V(P_n)$ by $f(\{a_i,b_j,b_k\})=(\{b_j,b_k\},a_i)$, for any $\{a_i,b_j,b_k\} \in W_3$. It is easy to see that $f$ is bijective. Moreover, by this correspondence we have that the subgraph induced by $W_3$ in $\Gamma_3(2P_n)$ is isomorphic to $\Gamma_2(P_n) \boxtimes P_n$.

For the case $A,B \in W_4$ with $A \neq B$, the proof is similar to the case $A,B \in W_3$ with $A \neq B$.
\end{proof}

By Lemma \ref{Lemma1}, for $n=4$, the $3$-token graph of $2P_n$ can be seen in Figure \ref{fig:3gamma2P4}.
\begin{figure}[H]
    \centering
    \resizebox{\textwidth}{!}{
    \begin{tikzpicture}[scale=1.5]
        \node[circle,fill,inner sep=1.5pt,label={above:$\{2,3,4\}$}] (234) at (1,11) {};
        \node[circle,fill,inner sep=1.5pt,label={above:$\{1,3,4\}$}] (134) at (1,8) {};
        \node[circle,fill,inner sep=1.5pt,label={below:$\{1,2,4\}$}] (124) at (1,5) {};
        \node[circle,fill,inner sep=1.5pt,label={below:$\{1,2,3\}$}] (123) at (1,2) {};

        \node[circle,fill,inner sep=1.5pt,label={above:$\{1,2,5\}$}] (125) at (3,2) {};
        \node[circle,fill,inner sep=1.5pt,label={above:$\{1,2,6\}$}] (126) at (3,5) {};
        \node[circle,fill,inner sep=1.5pt,label={below:$\{1,2,7\}$}] (127) at (3,8) {};
        \node[circle,fill,inner sep=1.5pt,label={above:$\{1,2,8\}$}] (128) at (3,11) {};
        \node[circle,fill,inner sep=1.5pt,label={above:$\{1,3,5\}$}] (135) at (5,2) {};
        \node[circle,fill,inner sep=1.5pt,label={below:$\{1,3,6\}$}] (136) at (5,5) {};
        \node[circle,fill,inner sep=1.5pt,label={below:$\{1,3,7\}$}] (137) at (5,8) {};
        \node[circle,fill,inner sep=1.5pt,label={above:$\{1,3,8\}$}] (138) at (5,11) {};
        \node[circle,fill,inner sep=1.5pt,label={below:$\{1,4,5\}$}] (145) at (6,3) {};
        \node[circle,fill,inner sep=1.5pt,label={below:$\{1,4,6\}$}] (146) at (6,6) {};
        \node[circle,fill,inner sep=1.5pt,label={above:$\{1,4,7\}$}] (147) at (6,9) {};
        \node[circle,fill,inner sep=1.5pt,label={above:$\{1,4,8\}$}] (148) at (6,12) {};
        \node[circle,fill,inner sep=1.5pt,label={below:$\{2,3,5\}$}] (235) at (7,1) {};
        \node[circle,fill,inner sep=1.5pt,label={below:$\{2,3,6\}$}] (236) at (7,4) {};
        \node[circle,fill,inner sep=1.5pt,label={above:$\{2,3,7\}$}] (237) at (7,7) {};
        \node[circle,fill,inner sep=1.5pt,label={above:$\{2,3,8\}$}] (238) at (7,10) {};
        \node[circle,fill,inner sep=1.5pt,label={below:$\{2,4,5\}$}] (245) at (8,2) {};
        \node[circle,fill,inner sep=1.5pt,label={below:$\{2,4,6\}$}] (246) at (8,5) {};
        \node[circle,fill,inner sep=1.5pt,label={above:$\{2,4,7\}$}] (247) at (8,8) {};
        \node[circle,fill,inner sep=1.5pt,label={above:$\{2,4,8\}$}] (248) at (8,11) {};
        \node[circle,fill,inner sep=1.5pt,label={below:$\{3,4,5\}$}] (345) at (10,2) {};
        \node[circle,fill,inner sep=1.5pt,label={below:$\{3,4,6\}$}] (346) at (10,5) {};
        \node[circle,fill,inner sep=1.5pt,label={above:$\{3,4,7\}$}] (347) at (10,8) {};
        \node[circle,fill,inner sep=1.5pt,label={above:$\{3,4,8\}$}] (348) at (10,11) {};

        \node[circle,fill,inner sep=1.5pt,label={above:$\{1,5,6\}$}] (156) at (12,2) {};
        \node[circle,fill,inner sep=1.5pt,label={above:$\{2,5,6\}$}] (256) at (12,5) {};
        \node[circle,fill,inner sep=1.5pt,label={below:$\{3,5,6\}$}] (356) at (12,8) {};
        \node[circle,fill,inner sep=1.5pt,label={above:$\{4,5,6\}$}] (456) at (12,11) {};
        \node[circle,fill,inner sep=1.5pt,label={above:$\{1,5,7\}$}] (157) at (14,2) {};
        \node[circle,fill,inner sep=1.5pt,label={below:$\{2,5,7\}$}] (257) at (14,5) {};
        \node[circle,fill,inner sep=1.5pt,label={below:$\{3,5,7\}$}] (357) at (14,8) {};
        \node[circle,fill,inner sep=1.5pt,label={above:$\{4,5,7\}$}] (457) at (14,11) {};
        \node[circle,fill,inner sep=1.5pt,label={below:$\{1,5,8\}$}] (158) at (15,3) {};
        \node[circle,fill,inner sep=1.5pt,label={below:$\{2,5,8\}$}] (258) at (15,6) {};
        \node[circle,fill,inner sep=1.5pt,label={above:$\{3,5,8\}$}] (358) at (15,9) {};
        \node[circle,fill,inner sep=1.5pt,label={above:$\{4,5,8\}$}] (458) at (15,12) {};
        \node[circle,fill,inner sep=1.5pt,label={below:$\{1,6,7\}$}] (167) at (16,1) {};
        \node[circle,fill,inner sep=1.5pt,label={below:$\{2,6,7\}$}] (267) at (16,4) {};
        \node[circle,fill,inner sep=1.5pt,label={above:$\{3,6,7\}$}] (367) at (16,7) {};
        \node[circle,fill,inner sep=1.5pt,label={above:$\{4,6,7\}$}] (467) at (16,10) {};
        \node[circle,fill,inner sep=1.5pt,label={below:$\{1,6,8\}$}] (168) at (17,2) {};
        \node[circle,fill,inner sep=1.5pt,label={below:$\{2,6,8\}$}] (268) at (17,5) {};
        \node[circle,fill,inner sep=1.5pt,label={above:$\{3,6,8\}$}] (368) at (17,8) {};
        \node[circle,fill,inner sep=1.5pt,label={above:$\{4,6,8\}$}] (468) at (17,11) {};
        \node[circle,fill,inner sep=1.5pt,label={below:$\{1,7,8\}$}] (178) at (19,2) {};
        \node[circle,fill,inner sep=1.5pt,label={below:$\{2,7,8\}$}] (278) at (19,5) {};
        \node[circle,fill,inner sep=1.5pt,label={above:$\{3,7,8\}$}] (378) at (19,8) {};
        \node[circle,fill,inner sep=1.5pt,label={above:$\{4,7,8\}$}] (478) at (19,11) {};

        \node[circle,fill,inner sep=1.5pt,label={above:$\{6,7,8\}$}] (678) at (21,11) {};
        \node[circle,fill,inner sep=1.5pt,label={above:$\{5,7,8\}$}] (578) at (21,8) {};
        \node[circle,fill,inner sep=1.5pt,label={below:$\{5,6,8\}$}] (568) at (21,5) {};
        \node[circle,fill,inner sep=1.5pt,label={below:$\{5,6,7\}$}] (567) at (21,2) {};
       
        \draw (123) -- (124);
        \draw (124) -- (134);
        \draw (134) -- (234);

        \draw (125) -- (126);
        \draw (126) -- (127);
        \draw (127) -- (128);
        \draw (125) -- (135);
        \draw (126) -- (136);
        \draw (127) -- (137);
        \draw (128) -- (138);
        \draw (135) -- (136);
        \draw (136) -- (137);
        \draw (137) -- (138);
        \draw (135) -- (145);
        \draw (135) -- (235);
        \draw (136) -- (146);
        \draw (136) -- (236);
        \draw (137) -- (147);
        \draw (137) -- (237);
        \draw (138) -- (148);
        \draw (138) -- (238);
        \draw (145) -- (146);
        \draw (146) -- (147);
        \draw (147) -- (148);
        \draw (235) -- (236);
        \draw (236) -- (237);
        \draw (237) -- (238);
        \draw (245) -- (235);
        \draw (245) -- (145);
        \draw (246) -- (236);
        \draw (246) -- (146);
        \draw (247) -- (237);
        \draw (247) -- (147);
        \draw (248) -- (238);
        \draw (248) -- (148);
        \draw (245) -- (246);
        \draw (246) -- (247);
        \draw (247) -- (248);
        \draw (245) -- (345);
        \draw (246) -- (346);
        \draw (247) -- (347);
        \draw (248) -- (348);
        \draw (345) -- (346);
        \draw (346) -- (347);
        \draw (347) -- (348);

        \draw (156) -- (256);
        \draw (256) -- (356);
        \draw (356) -- (456);
        \draw (157) -- (257);
        \draw (257) -- (357);
        \draw (357) -- (457);
        \draw (157) -- (158);
        \draw (157) -- (167);
        \draw (257) -- (258);
        \draw (257) -- (267);
        \draw (357) -- (358);
        \draw (357) -- (367);
        \draw (457) -- (458);
        \draw (457) -- (467);
        \draw (158) -- (258);
        \draw (258) -- (358);
        \draw (358) -- (458);
        \draw (167) -- (267);
        \draw (267) -- (367);
        \draw (367) -- (467);
        \draw (168) -- (158);
        \draw (168) -- (167);
        \draw (268) -- (258);
        \draw (268) -- (267);
        \draw (368) -- (358);
        \draw (368) -- (367);
        \draw (468) -- (458);
        \draw (468) -- (467);
        \draw (168) -- (268);
        \draw (268) -- (368);
        \draw (368) -- (468);
        \draw (168) -- (178);
        \draw (268) -- (278);
        \draw (368) -- (378);
        \draw (468) -- (478);
        \draw (178) -- (278);
        \draw (278) -- (378);
        \draw (378) -- (478);
        \draw (156) -- (157);
        \draw (256) -- (257);
        \draw (356) -- (357);
        \draw (456) -- (457);

        \draw (567) -- (568);
        \draw (568) -- (578);
        \draw (578) -- (678);
        
    \end{tikzpicture}}\caption{The Graph $\Gamma_3(2P_4)$}
    \label{fig:3gamma2P4}
    \end{figure}
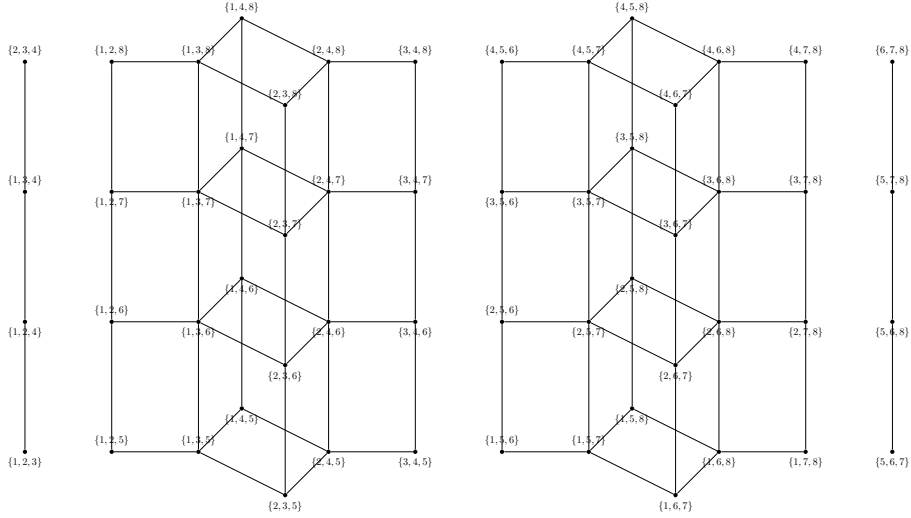

By examining Lemma \ref{Lemma1}, which establishes that  
\[
\Gamma_3(P_n \oplus P_n) = \Gamma_3(P_n) \oplus \Gamma_3(P_n) \oplus (\Gamma_2(P_n) \square P_n) \oplus (\Gamma_2(P_n) \square P_n),
\]  
we derive the following theorem, which provides a general characterization of the structure of the 3-token graph for any graph consisting of two components.\begin{theorem}\label{3token-disjoint}
        For any two graphs $G$ and $H$, it follows that        $$\Gamma_3(G \oplus H)=\Gamma_3(G) \oplus \Gamma_3(H) \oplus (\Gamma_2(G) \square H) \oplus (\Gamma_2(H) \square G).$$
    \end{theorem}
\begin{proof}
    By definition of the 3-token graph of $G$, we have
    $$V(\Gamma_3(G \oplus H))=\{ A \subseteq (V(G) \cup V(H)): |A|=3 \}.$$
    Let $V(G)=\{ a_i : i=1,2,\dots,n_1 \}$ and $V(H)= \{ b_i : i=1,2,\dots,n_2 \}$. We again form a partition $\{ W_1,W_2,W_3,W_4 \}$ on $V(\Gamma_3(G \oplus H))$, where
    $$W_1=\{A \subseteq V(\Gamma_3(G \oplus H)): A \subseteq V(G)\},$$
    $$W_2=\{A \subseteq V(\Gamma_3(G \oplus H)): A \subseteq V(H)\},$$
    $$W_3=\{A \subseteq V(\Gamma_3(G \oplus H)): |A \cap V(G)|=1 \text{ and } |A \cap V(H)|=2 \},$$
    $$W_4=\{A \subseteq V(\Gamma_3(G \oplus H)): |A \cap V(H)|=2 \text{ and } |A \cap V(H)|=1\}.$$

    It is clear that $\Gamma_3(G)$ and $\Gamma_3(H)$ are subgraphs of $\Gamma_3(G \oplus H)$, so that $\Gamma_3(G) \oplus \Gamma_3(H)$ is a subgraph of $\Gamma_3(G \oplus H)$. { Now we will show that each vertex in $\Gamma_3(G)$ or in $\Gamma_3(H)$ is not connected by any single vertex in $\Gamma_3(G \oplus H)$.} We know that $V(\Gamma_3(G))=W_1$ and $V(\Gamma_3(H)=W_2$. Clearly, for any $A \in W_1$ and for any $B \in W_2$, it follows that $A \cap B = \emptyset$, so that $|A \triangle B| = 6$. As a consequence, $AB \notin E(\Gamma_3(G \oplus H))$. Secondly, for any $A \in W_1$ and for any $C \in W_3 \cup W_4$, we consider the following cases:
    \begin{itemize}
        \item[(i)] Case: $C \in W_3$.\\
        Let $C=\{a_i,b_j,b_k\}$ with $i \in \{1,2,\dots,n_1\}$, $j,k \in \{1,2,\dots,n_2\}$ we got $|A \cap C| \leq 1 \}$, so $|A \triangle C| \geq 4$, clear that $AC \notin E(\Gamma_3(2P_n)).$
        \item[(ii)] Case: $C \in W_4$. \\Let $C=\{a_i, a_j, b_k\}$ with $i,j \in \{1,2,\dots,n_1\}$, $k \in \{1,2,\dots,n_2\}$. If $\{ a_i,a_j\} \nsubseteq A$, clear that $AC \notin E(\Gamma_3(G \oplus H))$. If $\{ a_i,a_j\} \subseteq A$, we have $A \triangle C = \{a_l,b_k\}$ with $l \in \{1,2,\dots,n_1\}$. Since $a_lb_k \notin (E(G) \cup E(H))=E(G \oplus H\}$, we get $AC \notin E(\Gamma_3(G \oplus H))$.
    \end{itemize}
    For any $B \in W_2$ and for any $C \in W_3 \cup W_4$,  $BC \notin E(\Gamma_3(G \oplus H))$.

    Now, let $C=\{ a_i,b_j,b_k \} \in W_3$ and $D=\{ a_m, a_n, b_p \} \in W_4$ where $i,m,n \in \{1,2,\dots,n_1\}$, $j,k,p \in \{1,2,\dots,n_2\}$.
    \begin{itemize}
        \item[(i)] If ($i=m$ or $i=n$) and ($p=j$ or $p=k$), we obtain $C \triangle D =\{a_t,b_u\}$ where $t \in \{m,n\}$ and $u \in \{j,k\}$. Hence, $CD \notin E(\Gamma_3(G \oplus H))$.
        \item[(ii)] If ($i \neq m$ and $i \neq n$) or ($p \neq j$ and $p \neq k$), obviously $|C \triangle D| \geq 4$, so that $CD \notin E(\Gamma_3(G \oplus H))$.
    \end{itemize}

    Now, consider the subgraphs of $\Gamma_3(G \oplus H)$ induced by $W_3$ and $W_4$. Let $A=\{a_i,b_j,b_k\} \in W_3$ and $B=\{a_m, b_n, b_p\} \in W_4$ with $A \neq B$ where $i,m \in \{1,2,\dots,n_1\}$, and $j,k,n,p \in \{1,2,\dots,n_2\}$. Clearly, $(i,j,k) \neq (m,n,p)$ and $(i,j,k) \neq (m,p,n)$. We will consider cases:
    \begin{itemize}
        \item [1. ] If $i=m$ and $j=n$ and $k \neq p$, then $A \triangle B = \{b_k,b_p\}$. Thus, $AB \in E(\Gamma_3(G \oplus H))$ if and only if $\{b_j,b_k\}\{b_n,b_p\} \in E(\Gamma_2(H))$.
        \item [2. ] If $i=m$ and $j=p$ and $k \neq n$, then $A \triangle B = \{b_k,b_n\}$. Thus, $AB \in E(\Gamma_3(G \oplus H))$ if and only if $\{b_j,b_k\}\{b_n,b_p\} \in E(\Gamma_2(H))$.
        \item [3. ] If $i=m$ and $k=n$ and $j \neq p$, then $A \triangle B = \{b_j,b_p\}$. Thus, $AB \in E(\Gamma_3(G \oplus H))$ if and only if $\{b_j,b_k\}\{b_n,b_p\} \in E(\Gamma_2(H))$.
        \item [4. ] If $i=m$ and $k=p$ and $j \neq n$, then $A \triangle B = \{b_j,b_n\}$. Thus, $AB \in E(\Gamma_3(G \oplus H))$ if and only if $\{b_j,b_k\}\{b_n,b_p\} \in E(\Gamma_2(H))$.
        \item [5. ] If $i=m$ and $j,k \neq n$ and $j,k\neq p$, then clearly $|A \triangle B| = 4$, so $AB \notin E(\Gamma_3(G \oplus H))$.
        \item [6. ] If $i \neq m$ and $\{j,k\}=\{n,p\}$, then $A \triangle B = \{a_i,a_m\}$. Thus, $AB \in E(\Gamma_3(G \oplus H))$ if and only if $a_ia_m \in E(G)$.
        \item [7. ] If $i \neq m$ and $j \in \{n,p\}$ and $k \notin \{n,p\}$, then clearly $|A \triangle B| = 4$, so that $AB \notin E(\Gamma_3(G \oplus H))$.
        \item [8. ] If $i \neq m$ and $k \in \{n,p\}$ and $j \notin \{n,p\}$, then clearly $|A \triangle B| = 4$, so that $AB \notin E(\Gamma_3(G \oplus H))$.
        \item [9. ] If $i \neq m$ and $j,k \neq n$ and $j,k\neq p$, then clearly $|A \triangle B| = 6$, so that $AB \notin E(\Gamma_3(G \oplus H))$.
    \end{itemize}
From these cases, for any $i,m \in \{1,2,\dots,n_1\}$, $j,k,n,p \in \{1,2,\dots,n_2\}$, we will have either $i=m$ implies $\{b_j,b_k\}\{b_n,b_p\} \in E(\Gamma_2(H))$ or $\{j,k\}=\{n,p\}$ implies $a_ia_m \in E(G)$.

Now, we construct a mapping $f: W_3 \to V(G) \times V(\Gamma_2(H))$ by $f(\{a_i,b_j,b_k\})=(\{b_j,b_k\},a_i)$, for any $\{a_i,b_j,b_k\} \in W_3$. It is easy to see that $f$ is bijective. Moreover, by this correspondence we have that the subgraph induced by $W_3$ in $\Gamma_3(G \oplus H)$ is isomorphic to $G \boxtimes \Gamma_2(H)$.

For the case $A,B \in W_4$ with $A \neq B$, the proof is similar to the case $A,B \in W_3$ with $A \neq B$.
\end{proof}

The following figure provides a visual representation of the 3-token graph of \(P_4 \oplus C_3\).
\begin{figure}[H]
    \centering
    \resizebox{\textwidth}{!}{
    \begin{tikzpicture}[scale=0.7]
        \node[circle,fill,inner sep=1.5pt,label={[scale=.7]left:$\{4,5,6\}$}] (456) at (1,1) {};
        \node[circle,fill,inner sep=1.5pt,label={[scale=.7]left:$\{4,5,7\}$}] (457) at (1,4) {};
        \node[circle,fill,inner sep=1.5pt,label={[scale=.7]left:$\{4,6,7\}$}] (467) at (1,9) {};
        \node[circle,fill,inner sep=1.5pt,label={[scale=.7]left:$\{5,6,7\}$}] (567) at (1,12) {};

        \node[circle,fill,inner sep=1.5pt,label={[scale=.7]below:$\{2,4,5\}$}] (245) at (4,1) {};
        \node[circle,fill,inner sep=1.5pt,label={[scale=.7]above:$\{3,4,5\}$}] (345) at (5,2) {};
        \node[circle,fill,inner sep=1.5pt,label={[scale=.7]below:$\{1,4,5\}$}] (145) at (6,1) {};
        \node[circle,fill,inner sep=1.5pt,label={[scale=.7]left:$\{2,4,6\}$}] (246) at (4,3) {};
        \node[circle,fill,inner sep=1.5pt,label={[scale=.7]below:$\{3,4,6\}$}] (346) at (5,4) {};
        \node[circle,fill,inner sep=1.5pt,label={[scale=.7]right:$\{1,4,6\}$}] (146) at (6,3) {};
        \node[circle,fill,inner sep=1.5pt,label={[scale=.7]below:$\{2,5,6\}$}] (256) at (2,6) {};
        \node[circle,fill,inner sep=1.5pt,label={[scale=.7]above:$\{1,5,6\}$}] (156) at (4,6) {};
        \node[circle,fill,inner sep=1.5pt,label={[scale=.7]above:$\{3,5,6\}$}] (356) at (3,7) {};
        \node[circle,fill,inner sep=1.5pt,label={[scale=.7]below:$\{2,4,7\}$}] (247) at (6,6) {};
        \node[circle,fill,inner sep=1.5pt,label={[scale=.7]above:$\{3,4,7\}$}] (347) at (7,7) {};
        \node[circle,fill,inner sep=1.5pt,label={[scale=.7]below:$\{1,4,7\}$}] (147) at (8,6) {};
        \node[circle,fill,inner sep=1.5pt,label={[scale=.7]below:$\{2,5,7\}$}] (257) at (4,9) {};
        \node[circle,fill,inner sep=1.5pt,label={[scale=.7]below:$\{1,5,7\}$}] (157) at (6,9) {};
        \node[circle,fill,inner sep=1.5pt,label={[scale=.7]above:$\{3,5,7\}$}] (357) at (5,10) {};
        \node[circle,fill,inner sep=1.5pt,label={[scale=.7]above:$\{2,6,7\}$}] (267) at (4,11) {};
        \node[circle,fill,inner sep=1.5pt,label={[scale=.7]above:$\{3,6,7\}$}] (367) at (5,12) {};
        \node[circle,fill,inner sep=1.5pt,label={[scale=.7]above:$\{1,6,7\}$}] (167) at (6,11) {};
        
        \node[circle,fill,inner sep=1.5pt,label={[scale=.7]left:$\{1,2,4\}$}] (124) at (9,1) {};
        \node[circle,fill,inner sep=1.5pt,label={[scale=.7]right:$\{1,3,4\}$}] (134) at (11,1) {};
        \node[circle,fill,inner sep=1.5pt,label={[scale=.7]below:$\{2,3,4\}$}] (234) at (10,2) {};
        \node[circle,fill,inner sep=1.5pt,label={[scale=.7]left:$\{1,2,5\}$}] (125) at (9,4) {};
        \node[circle,fill,inner sep=1.5pt,label={[scale=.7]right:$\{1,3,5\}$}] (135) at (11,4) {};
        \node[circle,fill,inner sep=1.5pt,label={[scale=.7]below:$\{2,3,5\}$}] (235) at (10,5) {};
        \node[circle,fill,inner sep=1.5pt,label={[scale=.7]above:$\{2,3,7\}$}] (237) at (10,12) {};
        \node[circle,fill,inner sep=1.5pt,label={[scale=.7]left:$\{1,2,7\}$}] (127) at (9,11) {};
        \node[circle,fill,inner sep=1.5pt,label={[scale=.7]right:$\{1,3,7\}$}] (137) at (11,11) {};
        \node[circle,fill,inner sep=1.5pt,label={[scale=.7]below:$\{2,3,6\}$}] (236) at (10,9) {};
        \node[circle,fill,inner sep=1.5pt,label={[scale=.7]left:$\{1,2,6\}$}] (126) at (9,8) {};
        \node[circle,fill,inner sep=1.5pt,label={[scale=.7]right:$\{1,3,6\}$}] (136) at (11,8) {};
        
        \node[circle,fill,inner sep=1.5pt,label={[scale=.7]above:$\{1,2,3\}$}] (123) at (14,6) {};
    
        \draw (456) -- (457);
        \draw (457) -- (467);
        \draw (467) -- (567);

        \draw (245) -- (145);
        \draw (245) -- (345);
        \draw (145) -- (345);
        \draw (245) -- (246);
        \draw (145) -- (146);
        \draw (345) -- (346);
        \draw (246) -- (256);
        \draw (346) -- (356);
        \draw (146) -- (156);
        \draw (246) -- (247);
        \draw (346) -- (347);
        \draw (146) -- (147);
        \draw (256) -- (257);
        \draw (356) -- (357);
        \draw (156) -- (157);
        \draw (247) -- (257);
        \draw (347) -- (357);
        \draw (147) -- (157);
        \draw (256) -- (156);
        \draw (246) -- (146);
        \draw (247) -- (147);
        \draw (256) -- (356);
        \draw (356) -- (156);
        \draw (246) -- (346);
        \draw (346) -- (146);
        \draw (247) -- (347);
        \draw (347) -- (147);
        \draw (257) -- (357);
        \draw (257) -- (157);
        \draw (357) -- (157);
        \draw (257) -- (267);
        \draw (367) -- (357);
        \draw (167) -- (157);
        \draw (267) -- (367);
        \draw (367) -- (167);
        \draw (267) -- (167);

        \draw (124) -- (234);
        \draw (124) -- (134);
        \draw (234) -- (134);
        \draw (125) -- (135);
        \draw (125) -- (235);
        \draw (235) -- (135);
        \draw (126) -- (136);
        \draw (126) -- (236);
        \draw (136) -- (236);
        \draw (127) -- (237);
        \draw (127) -- (137);
        \draw (137) -- (237);
        \draw (124) -- (125);
        \draw (134) -- (135);
        \draw (234) -- (235);
        \draw (125) -- (126);
        \draw (235) -- (236);
        \draw (135) -- (136);
        \draw (126) -- (127);
        \draw (236) -- (237);
        \draw (136) -- (137);                
    \end{tikzpicture}
    
    }
    \caption{The Graph $\Gamma_3(P_4 \oplus C_3)$}
    \label{fig:3gammaP4C3}
\end{figure}
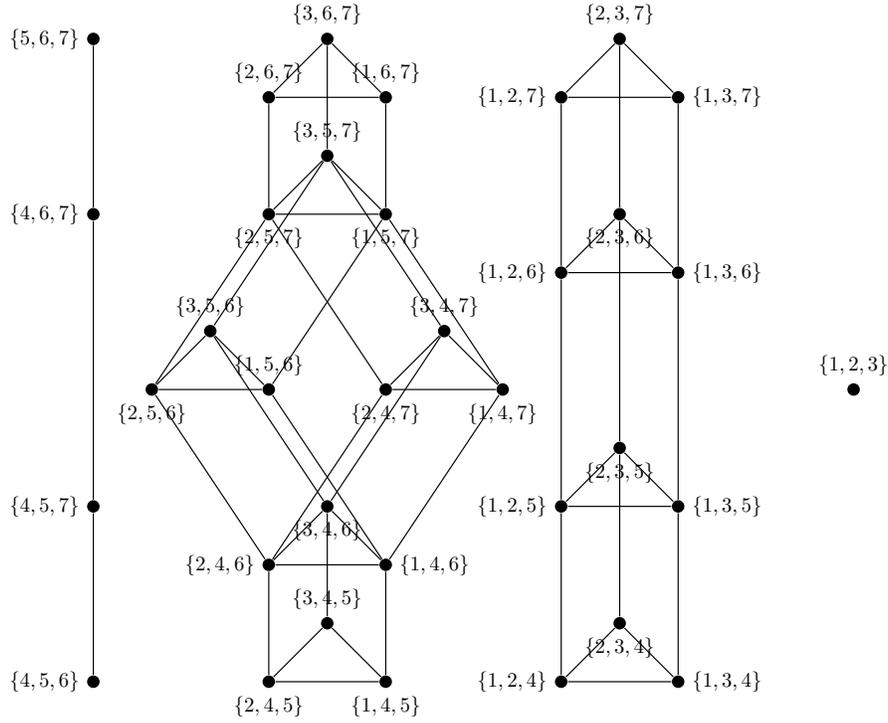


From Theorem \ref{3token-disjoint}, which describes the \(3\)-token graph of the disjoint union of two graphs, along with Theorem 2.1 as given in \cite{Susanti2023} describing the structure of the $2$-token graph of disconnected graphs, we can iteratively determine the \(3\)-token graph of a disconnected graph consisting of multiple components. By applying these results step by step, we can construct the \(3\)-token graph for more complex disconnected graphs with several individual components.
\begin{theorem}\label{Th-n-components}
Let $n\geq 2$ be arbitrary natural number and let $G_1,G_2,\ldots, G_n$ be arbitrary graphs. Then $$\Gamma_3(\bigoplus_{1}^{n} G_i)\cong \bigoplus_{i=1}^n\Gamma_3(G_i)\oplus \bigoplus_{1\leq i,j\leq n,i\neq j}(\Gamma_2(G_i)\square G_j)\oplus \bigoplus_{1\leq i<j<k\leq n}(G_i\square G_j\square G_k).$$
\end{theorem}
\begin{proof}
    We will prove by mathematical induction on the number of component $n$. By Theorem \ref{3token-disjoint}, the assertion is true for $n=2$. Assume that the assertion is true for arbitrary $n$. We proceed for $n+1$. By assumption and by Theorem 2.1 on the $2$-token graph of disconnected graphs given in \cite{Susanti2023}, we have
    $$\Gamma_3(\bigoplus_{1}^{n+1} G_i)\cong\Gamma_3(\bigoplus_{1}^n G_i \oplus G_{n+1})\cong\Gamma_3(\bigoplus_{i=1}^n G_i)\oplus \Gamma_3(G_{n+1})\oplus (\Gamma_2(\bigoplus_{i=1}^n G_i)\square G_{n+1})\oplus$$
    $$\bigoplus_{i=1}^n (G_i\square \Gamma_2(G_{n+1}))\cong\bigoplus_{i=1}^n \Gamma_3(G_i)\oplus \bigoplus_{1\leq i,j\leq n, i\neq j} (\Gamma_2(G_i)\square G_j) \oplus $$$$\bigoplus_{1\leq i<j<k\leq n} (G_i\square G_j\square G_k)\oplus \Gamma_3(G_{n+1})\oplus ((\bigoplus_{i=1}^n\Gamma_2(G_i)\oplus\bigoplus_{1\leq i<j\leq n} (G_i\square G_j))\square G_{n+1})$$ $$\oplus \bigoplus_{i=1}^n(G_i\square \Gamma_2(G_{n+1}))\cong\bigoplus_{i=1}^{n+1}\Gamma_3(G_i)\oplus \bigoplus_{1\leq i,j\leq n+1, i\neq j}(\Gamma_2(G_i)\square G_j)\oplus $$$$\bigoplus_{1\leq i<j<k\leq n+1}(G_i\square G_j\square G_k).$$
    
\end{proof}

As a direct consequence of Theorem \ref{Th-n-components}, we derive the following result concerning the number of components in the $3$-token graph of the disjoint union of graphs.

\begin{corollary}
    Let $n\geq 2$ be arbitrary natural number and let $G_1,G_2,\ldots, G_n$ be arbitrary graphs. The graph $\Gamma_3(\bigoplus_{i=1}^{n} G_i)$ is disconnected and has $n^2+{n\choose{3}}$ components.
\end{corollary}

Having derived the result for the \(3\)-token graph of disjoint union graphs with multiple components, we now shift our focus to a specific graph known as the cubical staircase graph. This graph possesses a unique structure and will, in subsequent analysis, be shown to be isomorphic to the \(3\)-token graph of a path. By exploring the properties of the cubical staircase graph, we aim to gain deeper insights into the characteristics of \(3\)-token graphs and their structural relationships.
\begin{definition}
    Let $n\geq 3$ be a natural number. The \textbf{cubical staircase} graph \(CS_n\) is a graph that is isomorphic to the graph \(G = (V(G), E(G))\) with
    \begin{align*}
        V(G)= &\{(i,j,k): 1 \leq i \leq n-2, 1\leq j \leq i, 1 \leq k \leq n-1-i\}\\
        \text{and}\\
        E(G)= &\{(i,j,k)(i,j,k+1): 1 \leq i \leq n-2, 1\leq j \leq i, 1 \leq k \leq n-2-i \} \\ &\cup \{(i,j,k)(i+1,j,k): 1 \leq i \leq n-3, 1\leq j \leq i, 1 \leq k \leq n-1-i\} \\ &\cup \{(i,j,k)(i,j+1,k): 1 \leq i \leq n-2, 1\leq j \leq i-1, 1 \leq k \leq n-1-i\}.    
    \end{align*}
\end{definition}
In the following example, we provide a particular cubical staircase graph $CS_n$ for $n=8$, as shown in Figure \ref{grafCS8}.

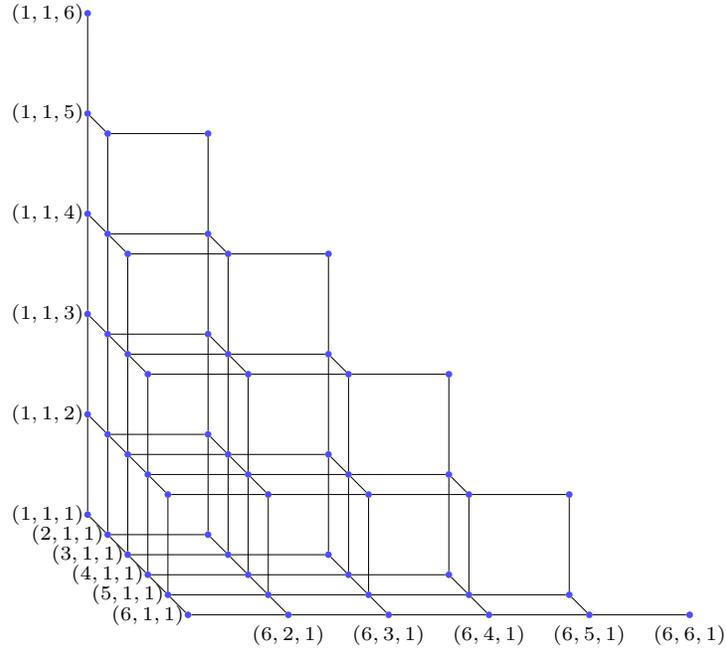
\begin{figure}[H]
		\begin{center} 
		\resizebox{0.8\textwidth}{!}{%
        \begin{tikzpicture}[scale=1, auto=center,
    every node/.style={font=\scriptsize},
    roundnode/.style={circle, fill=blue!70}]
				\node[roundnode] (a1) at (-13,14) {};  
				\node[roundnode] (a2) at (-13,9)  {}; 
				\node[roundnode] (a3) at (-13,4)  {};
				\node[roundnode] (a4) at (-13,-1) {};
                \node[roundnode] (a5) at (-13,-6)  {};
				\node[roundnode] (a6) at (-13,-11) {};
				\node[roundnode] (a7) at (-12,8) {};
                \node[roundnode] (a8) at (-12,3) {};  
				\node[roundnode] (a9) at (-12,-2) {}; 
				\node[roundnode] (a10) at (-12,-7) {};
				\node[roundnode] (a11) at (-12,-12) {};
				\node[roundnode] (a12) at (-7,8) {};
                \node[roundnode] (a13) at (-7,3) {};
				\node[roundnode] (a14) at (-7,-2) {};
				\node[roundnode] (a15) at (-7,-7) {};
                \node[roundnode] (a16) at (-7,-12) {}; 
				\node[roundnode] (a17) at (-11,2) {};
				\node[roundnode] (a18) at (-11,-3) {};
				\node[roundnode] (a19) at (-11,-8) {};
                \node[roundnode] (a20) at (-11,-13) {};
				\node[roundnode] (a21) at (-6,2) {};
				\node[roundnode] (a22) at (-6,-3) {};
                \node[roundnode] (a23) at (-6,-8) {};
				\node[roundnode] (a24) at (-6,-13) {};
				\node[roundnode] (a25) at (-1,2) {};
                \node[roundnode] (a26) at (-1,-3) {}; 
				\node[roundnode] (a27) at (-1,-8) {};
				\node[roundnode] (a28) at (-1,-13) {};
				\node[roundnode] (a29) at (-10,-4) {};
                \node[roundnode] (a30) at (-10,-9) {};
				\node[roundnode] (a31) at (-10,-14) {};
				\node[roundnode] (a32) at (-5,-4) {};
                \node[roundnode] (a33) at (-5,-9) {};
				\node[roundnode] (a34) at (-5,-14) {};
				\node[roundnode] (a35) at (0,-4) {};
                \node[roundnode] (a36) at (0,-9) {}; 
				\node[roundnode] (a37) at (0,-14) {};
				\node[roundnode] (a38) at (5,-4) {};
				\node[roundnode] (a39) at (5,-9) {};
                \node[roundnode] (a40) at (5,-14) {};
				\node[roundnode] (a41) at (-9,-10) {};
				\node[roundnode] (a42) at (-9,-15) {};
                \node[roundnode] (a43) at (-4,-10) {};
				\node[roundnode] (a44) at (-4,-15) {};
				\node[roundnode] (a45) at (1,-10) {};
                \node[roundnode] (a46) at (1,-15) {}; 
				\node[roundnode] (a47) at (6,-10) {};
				\node[roundnode] (a48) at (6,-15) {};
				\node[roundnode] (a49) at (11,-10) {};
                \node[roundnode] (a50) at (11,-15) {};
				\node[roundnode] (a51) at (-8,-16) {};
				\node[roundnode] (a52) at (-3,-16) {};
                \node[roundnode] (a53) at (2,-16) {};
				\node[roundnode] (a54) at (7,-16) {};
				\node[roundnode] (a55) at (12,-16) {};
                \node[roundnode] (a56) at (17,-16) {}; 
                \begin{scriptsize}
                    \draw[color=black] (-15,-11) node[scale=4]{$(1,1,1)$};
                    \draw[color=black] (-14,-12) node[scale=4]{$(2,1,1)$};
                    \draw[color=black] (-13,-13) node[scale=4]{$(3,1,1)$};
                    \draw[color=black] (-12,-14) node[scale=4]{$(4,1,1)$};
                    \draw[color=black] (-11,-15) node[scale=4]{$(5,1,1)$};
                    \draw[color=black] (-10,-16) node[scale=4]{$(6,1,1)$};
                    \draw[color=black] (-15,-6) node[scale=4]{$(1,1,2)$};
                    \draw[color=black] (-15,-1) node[scale=4]{$(1,1,3)$};
                    \draw[color=black] (-15,4) node[scale=4]{$(1,1,4)$};
                    \draw[color=black] (-15,9) node[scale=4]{$(1,1,5)$};
                    \draw[color=black] (-15,14) node[scale=4]{$(1,1,6)$};
                    \draw[color=black] (-3,-17) node[scale=4]{$(6,2,1)$};
                    \draw[color=black] (2,-17) node[scale=4]{$(6,3,1)$};
                    \draw[color=black] (7,-17) node[scale=4]{$(6,4,1)$};
                    \draw[color=black] (12,-17) node[scale=4]{$(6,5,1)$};
                    \draw[color=black] (17,-17) node[scale=4]{$(6,6,1)$};
                \end{scriptsize}
				\draw (a1) -- (a2);
                \draw (a2) -- (a3);
                \draw (a3) -- (a4);
                \draw (a4) -- (a5);
                \draw (a5) -- (a6);
                \draw (a7) -- (a8);
                \draw (a8) -- (a9);
                \draw (a9) -- (a10);
                \draw (a10) -- (a11);
                \draw (a12) -- (a13);
                \draw (a13) -- (a14);
                \draw (a14) -- (a15);
                \draw (a15) -- (a16);
                \draw (a17) -- (a18);
                \draw (a18) -- (a19);
                \draw (a19) -- (a20);
                \draw (a21) -- (a22);
                \draw (a22) -- (a23);
                \draw (a23) -- (a24);
                \draw (a25) -- (a26);
                \draw (a26) -- (a27);
                \draw (a27) -- (a28);
                \draw (a29) -- (a30);
                \draw (a30) -- (a31);
                \draw (a32) -- (a33);
                \draw (a33) -- (a34);
                \draw (a35) -- (a36);
                \draw (a36) -- (a37);
                \draw (a38) -- (a39);
                \draw (a39) -- (a40);
                \draw (a41) -- (a42);
                \draw (a43) -- (a44);
                \draw (a45) -- (a46);
                \draw (a47) -- (a48);
                \draw (a49) -- (a50);
                \draw (a51) -- (a52);
                \draw (a52) -- (a53);
                \draw (a53) -- (a54);
                \draw (a54) -- (a55);
                \draw (a55) -- (a56);
                \draw (a7) -- (a12);
                \draw (a8) -- (a13);
                \draw (a9) -- (a14);
                \draw (a10) -- (a15);
                \draw (a11) -- (a16);
                \draw (a17) -- (a21);
                \draw (a21) -- (a25);
                \draw (a18) -- (a22);
                \draw (a22) -- (a26);
                \draw (a19) -- (a23);
                \draw (a23) -- (a27);
                \draw (a20) -- (a24);
                \draw (a24) -- (a28);
                \draw (a29) -- (a32);
                \draw (a32) -- (a35);
                \draw (a35) -- (a38);
                \draw (a30) -- (a33);
                \draw (a33) -- (a36);
                \draw (a36) -- (a39);
                \draw (a31) -- (a34);
                \draw (a34) -- (a37);
                \draw (a37) -- (a40);
                \draw (a42) -- (a44);
                \draw (a44) -- (a46);
                \draw (a46) -- (a48);
                \draw (a48) -- (a50);
                \draw (a41) -- (a43);
                \draw (a43) -- (a45);
                \draw (a45) -- (a47);
                \draw (a47) -- (a49);
                \draw (a2) -- (a7);
                \draw (a3) -- (a8);
                \draw (a4) -- (a9);
                \draw (a5) -- (a10);
                \draw (a6) -- (a11);
                \draw (a8) -- (a17);
                \draw (a9) -- (a18);
                \draw (a10) -- (a19);
                \draw (a11) -- (a20);
                \draw (a13) -- (a21);
                \draw (a14) -- (a22);
                \draw (a15) -- (a23);
                \draw (a16) -- (a24);
                \draw (a18) -- (a29);
                \draw (a19) -- (a30);
                \draw (a20) -- (a31);
                \draw (a22) -- (a32);
                \draw (a23) -- (a33);
                \draw (a24) -- (a34);
                \draw (a26) -- (a35);
                \draw (a27) -- (a36);
                \draw (a28) -- (a37);
                \draw (a30) -- (a41);
                \draw (a31) -- (a42);
                \draw (a33) -- (a43);
                \draw (a34) -- (a44);
                \draw (a36) -- (a45);
                \draw (a37) -- (a46);
                \draw (a39) -- (a47);
                \draw (a40) -- (a48);
                \draw (a42) -- (a51);
                \draw (a44) -- (a52);
                \draw (a46) -- (a53);
                \draw (a48) -- (a54);
                \draw (a50) -- (a55); 
                \end{tikzpicture}%
                } 
            \caption{The Graph $CS_8$}\label{grafCS8}
		\end{center}
	\end{figure}

For $n=4,5,6,7$, the graphS $CS_n$ are described in Figure \ref{grafCS4567}.

\begin{figure}[H]
		\begin{center} 
		\resizebox{0.9\textwidth}{!}{
        \begin{tikzpicture}  
		[scale=1,auto=center,every node/.style={circle, fill=blue!70}] 
				\node (a1) at (-13,14){};  
				\node (a2) at (-13,9) {}; 
                \node (a3) at (-12,8){};
                \node (a4) at (-7,8) {};

				\draw (a1) -- (a2);
                \draw (a3) -- (a4);
                \draw (a2) -- (a3);
			\end{tikzpicture}
                \begin{tikzpicture}  
				[scale=1,auto=center,every node/.style={circle, fill=blue!70}] 
				\node (a1) at (-13,14) {};  
				\node (a2) at (-13,9)  {}; 
                \node (a3) at (-12,8) {};
                \node (a4) at (-7,8) {};
				\node (a5) at (-13,4)  {};
                \node (a6) at (-12,3) {};  
                \node (a7) at (-7,3) {};
                \node (a8) at (-11,2) {};
                \node (a9) at (-6,2) {};
                \node (a10) at (-1,2) {};

				\draw (a1) -- (a2);
                \draw (a2) -- (a5);
                \draw (a3) -- (a6);
                \draw (a4) -- (a7);
                \draw (a3) -- (a4);
                \draw (a6) -- (a7);
                \draw (a8) -- (a9);
                \draw (a9) -- (a10);
                \draw (a2) -- (a3);
                \draw (a5) -- (a6);
                \draw (a6) -- (a8);
                \draw (a7) -- (a9);
			\end{tikzpicture}
                \begin{tikzpicture}  
				[scale=1,auto=center,every node/.style={circle, fill=blue!70}] 
				\node (a1) at (-13,14) {};  
				\node (a2) at (-13,9)  {}; 
                \node (a3) at (-12,8) {};
                \node (a4) at (-7,8) {};
				\node (a5) at (-13,4)  {};
                \node (a6) at (-12,3) {};  
                \node (a7) at (-7,3) {};
                \node (a8) at (-11,2) {};
                \node (a9) at (-6,2) {};
                \node (a10) at (-1,2) {};
                \node (a11) at (-13,-1) {};
                \node (a12) at (-12,-2) {}; 
                \node (a13) at (-7,-2) {};
                \node (a14) at (-11,-3) {};
				\node (a15) at (-6,-3) {};
                \node (a16) at (-1,-3) {}; 
                \node (a17) at (-10,-4) {};
                \node (a18) at (-5,-4) {};
                \node (a19) at (0,-4) {};
                \node (a20) at (5,-4) {};

				\draw (a1) -- (a2);
                \draw (a2) -- (a5);
                \draw (a5) -- (a11);
                \draw (a3) -- (a6);
                \draw (a6) -- (a12);
                \draw (a4) -- (a7);
                \draw (a7) -- (a13);
                \draw (a8) -- (a14);
                \draw (a9) -- (a15);
                \draw (a10) -- (a16);
                \draw (a3) -- (a4);
                \draw (a6) -- (a7);
                \draw (a12) -- (a13);
                \draw (a8) -- (a9);
                \draw (a9) -- (a10);
                \draw (a14) -- (a15);
                \draw (a15) -- (a16);
                \draw (a17) -- (a18);
                \draw (a18) -- (a19);
                \draw (a19) -- (a20);
                \draw (a2) -- (a3);
                \draw (a5) -- (a6);
                \draw (a11) -- (a12);
                \draw (a6) -- (a8);
                \draw (a12) -- (a14);
                \draw (a7) -- (a9);
                \draw (a13) -- (a15);
                \draw (a14) -- (a17);
                \draw (a15) -- (a18);
                \draw (a16) -- (a19);
			\end{tikzpicture}
                \begin{tikzpicture}  
				[scale=1,auto=center,every node/.style={circle, fill=blue!70}] 
				\node (a1) at (-13,14) {};  
				\node (a2) at (-13,9)  {}; 
                \node (a3) at (-12,8) {};
                \node (a4) at (-7,8) {};
				\node (a5) at (-13,4)  {};
                \node (a6) at (-12,3) {};  
                \node (a7) at (-7,3) {};
                \node (a8) at (-11,2) {};
                \node (a9) at (-6,2) {};
                \node (a10) at (-1,2) {};
                \node (a11) at (-13,-1) {};
                \node (a12) at (-12,-2) {}; 
                \node (a13) at (-7,-2) {};
                \node (a14) at (-11,-3) {};
				\node (a15) at (-6,-3) {};
                \node (a16) at (-1,-3) {}; 
                \node (a17) at (-10,-4) {};
                \node (a18) at (-5,-4) {};
                \node (a19) at (0,-4) {};
                \node (a20) at (5,-4) {};
                \node (a21) at (-13,-6)  {};
                \node (a22) at (-12,-7) {};
                \node (a23) at (-7,-7) {};
                \node (a24) at (-11,-8) {};
                \node (a25) at (-6,-8) {};
                \node (a26) at (-1,-8) {};
                \node (a27) at (-10,-9) {};
                \node (a28) at (-5,-9) {};
                \node (a29) at (0,-9) {}; 
                \node (a30) at (5,-9) {};
                \node (a31) at (-9,-10) {};
                \node (a32) at (-4,-10) {};
                \node (a33) at (1,-10) {};
                \node (a34) at (6,-10) {};
                \node (a35) at (11,-10) {};

				\draw (a1) -- (a2);
                \draw (a2) -- (a5);
                \draw (a5) -- (a11);
                \draw (a11) -- (a21);
                \draw (a3) -- (a6);
                \draw (a6) -- (a12);
                \draw (a12) -- (a22);
                \draw (a4) -- (a7);
                \draw (a7) -- (a13);
                \draw (a13) -- (a23);
                \draw (a8) -- (a14);
                \draw (a14) -- (a24);
                \draw (a9) -- (a15);
                \draw (a15) -- (a25);
                \draw (a10) -- (a16);
                \draw (a16) -- (a26);
                \draw (a17) -- (a27);
                \draw (a18) -- (a28);
                \draw (a19) -- (a29);
                \draw (a20) -- (a30);
                \draw (a3) -- (a4);
                \draw (a6) -- (a7);
                \draw (a12) -- (a13);
                \draw (a22) -- (a23);
                \draw (a8) -- (a9);
                \draw (a9) -- (a10);
                \draw (a14) -- (a15);
                \draw (a15) -- (a16);
                \draw (a24) -- (a25);
                \draw (a25) -- (a26);
                \draw (a17) -- (a18);
                \draw (a18) -- (a19);
                \draw (a19) -- (a20);
                \draw (a27) -- (a28);
                \draw (a28) -- (a29);
                \draw (a29) -- (a30);
                \draw (a31) -- (a32);
                \draw (a32) -- (a33);
                \draw (a33) -- (a34);
                \draw (a34) -- (a35);
                \draw (a2) -- (a3);
                \draw (a5) -- (a6);
                \draw (a11) -- (a12);
                \draw (a21) -- (a22);
                \draw (a6) -- (a8);
                \draw (a12) -- (a14);
                \draw (a22) -- (a24);
                \draw (a7) -- (a9);
                \draw (a13) -- (a15);
                \draw (a23) -- (a25);
                \draw (a14) -- (a17);
                \draw (a24) -- (a27);
                \draw (a15) -- (a18);
                \draw (a25) -- (a28);
                \draw (a16) -- (a19);
                \draw (a26) -- (a29);
                \draw (a27) -- (a31);
                \draw (a28) -- (a32);
                \draw (a29) -- (a33);
                \draw (a30) -- (a34);
			\end{tikzpicture}}
			\caption{Graphs $CS_4$, $CS_5$, $CS_6$, and $CS_7$}\label{grafCS4567}
		\end{center}
	\end{figure}
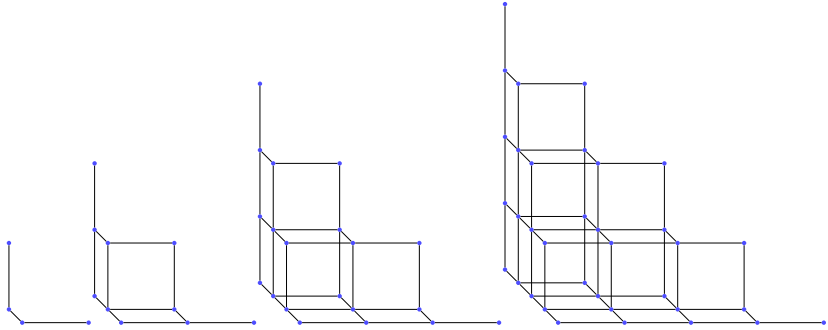

The following lemma provides the distance between any two distinct vertices in the graph \(CS_n\), offering a clearer understanding of the graph's structural properties.

\begin{lemma}\label{lemmadCSn}
    For any natural number $n\geq 4$, let $(a_1,a_2,a_3),(b_1,b_2,b_3) \in V(CS_n)$ be arbitrary. Then
    \begin{align*}
        d_{CS_n}\left((a_1,a_2,a_3),(b_1,b_2,b_3) \right)=|a_1-b_1|+|a_2-b_2|+|a_3-b_3|.
    \end{align*}
\end{lemma}
\begin{proof}
    Consider the graph $CS_n$ for $n \geq 4$. Let 
    $(a_1, a_2, a_3), (b_1, b_2, b_3) \in V(CS_n).$ We observe for some cases.
    \begin{itemize}
        \item[(i)] Cases $a_1 \leq b_1$, there exists a path $$(a_1,a_2,a_3)-(a_1+1,a_2,a_3)-(a_1+2,a_2,a_3)-\cdots-(b_1,a_2,a_3).$$
        \item[(ii)] Cases $a_1 > b_1$, there exist a path
        $$(a_1,a_2,a_3)-(a_1-1,a_2,a_3)-(a_1-2,a_2,a_3)-\cdots-(b_1,a_2,a_3).$$
    \end{itemize}
    So, from Cases (i) and (ii), we get $d_{CS_n}\left((a_1,a_2,a_3),(b_1,a_2,a_3) \right)=|a_1-b_1|$.
    Now, we observe for following cases.
    \begin{itemize}
        \item[(iii)] Cases $a_2 \leq b_2$. There exists a path $$(b_1,a_2,a_3)-(b_1,a_2+1,a_3)-(b_1,a_2+2,a_3)-\cdots-(b_1,b_2,a_3).$$
        \item[(iv)] Cases $a_2 > b_2$. There exists a path
        $$(b_1,a_2,a_3)-(b_1,a_2-1,a_3)-(b_1,a_2-2,a_3)-\cdots-(b_1,b_2,a_3).$$
    \end{itemize}
    Consequently, we get $d_{CS_n}\left((b_1,a_2,a_3),(b_1,b_2,a_3) \right)=|a_2-b_2|$.
    Similar to Cases (i) and (ii) together with Cases (iii) and (iv), we have $$d_{CS_n}\left((b_1,b_2,a_3),(b_1,b_2,b_3) \right)=|a_3-b_3|$$ and as a result, we obtain   $$d_{CS_n}\left((a_1,a_2,a_3),(b_1,b_2,b_3) \right) \leq |a_1-b_1|+|a_2-b_2|+|a_3-b_3|.$$
    Since $(a_1,a_2,a_3)(b_1,b_2,b_3) \in E(CS_n)$ if and only if either $a_1=b_1, a_2=b_2$, and $b_3=a_3+1$; or $a_1=b_1, a_3=b_3$, and $b_2=a_2+1$; or $a_2=b_2, a_3=b_3$, and $b_1=a_1+1$, we get $d_{CS_n}\left((a_1,a_2,a_3),(b_1,b_2,b_3)\right) \geq |a_1-b_1|+|a_2-b_2|+|a_3-b_3|.$
\end{proof}

The lemma below demonstrates that the graph \( CS_n \) contains no triangles.
\begin{lemma}\label{lemma 3-token don't have C_3}
    For every $n\geq 4$, graph $CS_n$ has no triangle.
\end{lemma}
\begin{proof}
    Let $(a,b,c) \in V(CS_n)$ where $n \geq 4$. We know that
    \begin{align*}
    (a,b,c)(a+1,b,c), &(a,b,c)(a,b+1,c), (a,b,c)(a,b,c+1), (a,b,c)(a-1,b,c),\\ &(a,b,c)(a,b-1,c), (a,b,c)(a,b,c-1) \in E(CS_n). 
    \end{align*}
    In this case, we only consider for 
    \begin{align*}
        &(a+1,b,c)(a+2,b,c), (a+1,b,c)(a+1,b+1,c), (a+1,b,c)(a+1,b,c+1),\\ &(a+1,b,c)(a+1,b-1,c), (a+1,b,c)(a+1,b,c-1) \in E(CS_n)
    \end{align*}
    \begin{align*}
        &(a,b+1,c)(a+1,b+1,c), (a,b+1,c)(a,b+2,c), (a,b+1,c)(a,b+1,c+1),\\ & (a,b+1,c)(a-1,b+1,c),(a,b+1,c)(a,b+1,c-1) \in E(CS_n)
    \end{align*}
    \begin{align*}
        &(a,b,c+1)(a+1,b,c+1), (a,b,c+1)(a,b+1,c+1), (a,b,c+1)(a,b,c+2), \\ &(a,b,c+1)(a-1,b,c+1),(a,b,c+1)(a,b-1,c+1)\in E(CS_n)
    \end{align*}
    \begin{align*}
        &(a-1,b,c)(a-1,b+1,c), (a-1,b,c)(a-1,b,c+1), (a-1,b,c)(a-1,b,c), \\ &(a-1,b,c)(a-1,b-1,c),(a-1,b,c)(a-1,b,c-1)\in E(CS_n)
    \end{align*}
    \begin{align*}
        &(a,b-1,c)(a+1,b-1,c), (a,b-1,c)(a,b-1,c+1), (a,b-1,c)(a-1,b-1,c), \\ &(a,b-1,c)(a,b-2,c), (a,b-1,c)(a,b-1,c-1) \in E(CS_n).
    \end{align*}
    Also we have
    \begin{align*}
        &(a,b,c-1)(a+1,b,c-1), (a,b,c-1)(a,b+1,c-1), (a,b,c-1)(a-1,b,c-1), \\ &(a,b,c-1)(a,b-1,c-1), (a,b,c-1)(a,b,c-2) \in E(CS_n).
    \end{align*}
    Since \begin{align*}
        &(a+2,b,c)(a,b,c), (a+1,b+1,c)(a,b,c), (a+1,b,c+1) (a,b,c), \\ &(a+1,b-1,c)(a,b,c), (a+1,b,c-1)(a,b,c), (a,b+2,c)(a,b,c), \\ &(a,b+1,c+1)(a,b,c), (a-1,b+1,c)(a,b,c), (a,b+1,c-1)(a,b,c),  \\ &(a,b,c+2)(a,b,c), (a-1,b,c+1)(a,b,c),(a,b-1,c+1)(a,b,c), \\ &(a-2,b,c)(a,b,c), (a-1,b-1,c)(a,b,c), (a-1,b,c-1)(a,b,c), \\ &(a,b-2,c)(a,b,c), (a,b-1,c-1)(a,b,c), (a,b+1,c-1)(a,b,c), \\ &(a,b,c+2)(a,b,c) \notin E(CS_n)
    \end{align*} so we have no any subgraph that isomorphic to graph $C_3$.  
\end{proof}

The following result establishes that the graph \( CS_n \) is either one-colorable or two-colorable.

\begin{theorem}\label{chiCS_n}
    For any positive integer $n\geq 3$. Graph $CS_n$ has chromatic number 
    \begin{align*}
        \chi(CS_n) = \begin{cases}
            1, & \text{ if } n=3,\\
            2, & \text{ if } n \geq 3.
            \end{cases}
    \end{align*}
\end{theorem}
\begin{proof}
    For $n=3$, we have that the graph $CS_n$ is isomorphic to $K_1$. Then, we get $\chi(CS_n)=1$. Now, let $n \geq 4$, for any vertex $(a,b,c) \in V(CS_n)$, we color the vertex with the first color if $a+b+c$ is even and we color the vertex with second color if $a+b+c$ is odd.
    Now, for all $(a_1,a_2,a_3)(b_1,b_2,b_3) \in E(CS_n)$, we get
    \begin{itemize}
        \item[(i)] Case 1: $a_1=b_1$, $a_2=b_2$, and $b_3=a_3+1$. 
        \item[(ii)] Case 2: $a_1=b_1$, $a_3=b_3$, and $b_2=a_2+1$.
        \item[(iii)] Case 3: $a_2=b_2$, $a_3=b_3$, and $b_1=a_1+1$.
    \end{itemize}
    For all cases, we have that $a_1+a_2+a_3$ and $b_1+b_2+b_3=a_1+a_2+a_3+1$ are odd and even respectively or even and odd respectively.
    Therefore, vertex $(a_1,a_2,a_3)$ and vertex $(b_1,b_2,b_3)$ have different colors.
\end{proof}

In these three following theorems, we determine the clique number \( \omega(CS_n) \),  the diameter and the independence number of the graph \( CS_n \) when \( n \geq 3 \), respectively.

\begin{theorem}\label{bilangan kromatik dan clique}
    For any positive integer $n\geq 3$, the graph $CS_n$ has the following properties:
    \begin{align*}
        \omega(CS_n)= \begin{cases}
            1, & \text{ if } n=3,\\
            2, & \text{ if } n > 3.
        \end{cases}
    \end{align*}
\end{theorem}
\begin{proof}
    Let $n \geq 3$. 
    For $n=3$, we know that $CS_3$ is a $K_1$ so $\omega(CS_3)=1$. Then, for $n \geq 4$, it is clear that $\omega(CS_n) \geq 2$. From Lemma \ref{lemma 3-token don't have C_3} we know that the $CS_n$ does not have a triangle. We conclude that $CS_n$ does not have $C_3$ as an induced subgraph. In other word,  $\omega(CS_n) < 3$. Thus,  $\omega(CS_n) = 2$.
\end{proof}

\begin{theorem}\label{teo_diamCSn}
    For any $n \geq 4$, the graph $CS_n$ has the following property
    \begin{align*}
        diam(CS_n)=3(n-3).
    \end{align*}
\end{theorem}
\begin{proof}
    Consider the vertices $(1,1,n-2), (n-2, n-2, 1) \in V(CS_n)$. Using Lemma \ref{lemmadCSn}, we obtain: $$d_{CS_n}\left((1,1,n-2),(n-2, n-2, 1)\right) = 3(n-3).$$

    Now, for any $(a_1,b_1,c_1),(a_2,b_2,c_2) \in V(CS_n)$, the following holds:
    \begin{align*}
        d_{CS_n}\left((a_1,b_1,c_1),(a_2,b_2,c_2) \right)&=\left|a_1-a_2\right|+\left|b_1-b_2\right|+\left|c_1-c_2\right|\\
        &\leq (n-2-1)+(n-2-1)+(n-2-1)\\&=3(n-3). 
    \end{align*}
\end{proof}

\begin{theorem}\label{teo_alpha_CSn}
    For any positive integer $n \geq 3$, the graph $CS_n$ satisfies that
    \begin{align*}
    \alpha(CS_n) = \begin{cases}
            \frac{n^3-3n^2+2n}{12}, & \text{ if } n \text{ is even},\\
            \frac{n^3-3n^2+5n-3}{12}, & \text{ if } n \text{ is odd}.
        \end{cases}
    \end{align*}
\end{theorem}
\begin{proof}
    If $n$ is odd, we can construct an
    independent set in graph $CS_n$,
    \begin{align*}
        A=\{(i,j,k) \in V(CS_n): i+j+k \text{ is odd}\}
    \end{align*}
    we get $|A|=\frac{n^3-3n^2+5n-3}{12}$ and it is clear that $A$ is the only independent set with cardinality $\frac{n^3-3n^2+5n-3}{12}$ in $CS_n$.  From $A$, if we add one vertex $(i,j,k) \in V(CS_n) \setminus A$, i.e. $i+j+k$ is even.
    \begin{itemize}
        \item[(i)] If $i<n-2$, we can found $(i+1,j,k) \in A$ such that $(i,j,k)(i+1,j,k) \in E(CS_n)$. 
        \item[(ii)] If $i=n-2$, we can found $(i-1,j,k) \in A$ such that $(i-1,j,k)(i,j,k) \in E(CS_n)$. 
    \end{itemize}
    So, $A \cup \{(i,j,k)\}$ is not an independent set. If $n$ is even, we have
    \begin{align*}
        B = \{(i,j,k) \in V(CS_n): i+j+k \text{ is even}\}\\
        C = \{(i,j,k) \in V(CS_n): i+j+k \text{ is odd}\} 
    \end{align*} is an independent set in graph $CS_n$.
    We have that $|B|=|C|=\frac{n^3-3n^2+2n}{12}$. It is clear that the independent sets with cardinality $\frac{n^3-3n^2+2n}{12}$ in $CS_n$ are only the sets $B$ and $C$. From $B$, if we add one vertex $(i,j,k) \in V(CS_n) \setminus B$, i.e. $i+j+k$ is odd.
    \begin{itemize}
        \item[(i)] If $i<n-2$, we can found $(i+1,j,k) \in B$ such that $(i,j,k)(i+1,j,k) \in E(CS_n)$. 
        \item[(ii)] If $i=n-2$, we can found $(i-1,j,k) \in B$ such that $(i-1,j,k)(i,j,k) \in E(CS_n)$. 
    \end{itemize}
    So, $B \cup \{(i,j,k)\}$ is not an independent set. Then, the case for the set $C$ will be the same as the case when \(n\) is odd, so it follows that \(C \cup \{(i,j,k)\}\) is not an independent set.
\end{proof}

In the study of token graphs, understanding the structural relationships between graphs provides valuable insights. In the following theorem, we demonstrate that the $3$-token graph of the path graph \( P_n \) is isomorphic to the graph \( CS_n \) for \( n \geq 4 \), establishing a clear connection between these two graph classes.

\begin{theorem}\label{isomorfisGamma3Pn}
    Let $P_n$, 
    $n \geq 3$ be path graph. Then $\Gamma_3(P_n) \cong CS_n$.
\end{theorem}
\begin{proof}
    For $n=3$, we know that  $$\Gamma_3(P_3) \cong K_1 \cong CS_3.$$
    Now, let $V(P_n)=\{x_i: 1\leq i \leq n\}$ and $E(P_n)=\{x_ix_{i+1}: 1 \leq i \leq n-1\}$ where $n \geq 4$. First, we have that
    \begin{align*}
        \left|V(CS_n)\right|&=\sum_{i=1}^{n-2} (i)(n-1-i)\\
        &= \sum_{i=1}^{n-2} \left((n-1)i-i^2 \right)\\
        &= \frac{n(n-1)(n-2)}{6}\\
        &= \binom{n}{3}=\left| V(\Gamma_3(P_n))\right|.
    \end{align*}
    Then for any $A=\{x_i,x_j,x_k\}, B=\{x_r,x_s,x_t\} \in V(\Gamma_3(P_n))$, $AB \in E(\Gamma_3(P_n))$ if and only if $A \triangle B =\{x_p,x_{p+1}\}$ for some $p \in \{i,j,k,r,s,t\}$. 
    Now, we construct a mapping
    $$\psi: V(\Gamma_3(P_n)) \to CS_n$$
    by $\psi\left(\{x_i,x_j,x_k\}\right)=(j-1,i, n+1-k\}$ for any $i<j<k$ with $1\leq i <j<k \leq n$. Let $\{x_i,x_j,x_k\},\{x_r,x_s,x_t\} \in V(\Gamma_3(P_n))$. If $\{x_i,x_j,x_k\}\{x_r,x_s,x_t\} \in E(\Gamma_3(P_n))$, then we have 
    $\{x_i,x_j,x_k\}\triangle \{x_r,x_s,x_t\}=\{x_p,x_{p+1}\}$ for some $x_px_{p+1} \in E(P_n)$. WLOG, let $x_i=x_r, x_j=x_s$, i.e. $i=r, j=s$ and $k=p, t=p+1$.
    \begin{itemize}
        \item [(i)] If $i<j<k$, then
        \begin{align*}
            \psi(\{x_i,x_j,x_k\})&\psi(\{x_r,x_s,x_t\})=(j-1,i,n+1-k)(s-1,r,n+1-t)\\&=(j-1,i,n-p+1)(j-1,i,n-p) \in E(CS_n).
        \end{align*}
    \item [(ii)] If $i<k<j$, then
        \begin{align*}
        \psi(\{x_i,x_j,x_k\})&\psi(\{x_r,x_s,x_t\})=(k-1,i,n+1-j)(t-1,r,n+1-s)\\&=(p-1,i,n+1-j)(p-1+1,i,n+1-j) \in E(CS_n).
        \end{align*}
    \item [(iii)] If $j<i<k$, then
        \begin{align*}
        \psi(\{x_i,x_j,x_k\})&\psi(\{x_r,x_s,x_t\})=(i-1,j,n+1-k)(r-1,j,n+1-t)\\&=(i-1,j,n-p+1)(i-1,s,n-p) \in E(CS_n).
        \end{align*}
    \item [(iv)] If $j<k<i$, then
        \begin{align*}
        \psi(\{x_i,x_j,x_k\})&\psi(\{x_r,x_s,x_t\})=(k-1,j,n+1-i)(t-1,s,n+1-r)\\&=(p-1,j,n+1-i)(p,j,n+1-i) \in E(CS_n).
        \end{align*}
    \item [(v)] If $k<i<j$, then
        \begin{align*}
        \psi(\{x_i,x_j,x_k\})&\psi(\{x_r,x_s,x_t\})=(i-1,k,n+1-j)(r-1,t,n+1-s)\\&=(i-1,p,n+1-j)(i-1,p+1,n+1-j) \in E(CS_n).
        \end{align*}
    \item [(vi)] If $k<j<i$, then
        \begin{align*}
        \psi(\{x_i,x_j,x_k\})&\psi(\{x_r,x_s,x_t\})=(j-1,k,n+1-i)(s-1,t,n+1-r)\\&=(j-1,p,n+1-i)(j-1,p+1,n+1-i) \in E(CS_n).
        \end{align*}
    \end{itemize}

    Conversely, let $\psi(\{x_i,x_j,x_k\})\psi(\{x_r,x_s,x_t\}) \in E(CS_n)$. Let $i<j<k$ and $r<s<t$. Then,
    $$(j-1,i,n+1-k)(s-1,r,n+1-t) \in E(CS_n).$$
    It means that either $j-1=s-1, i=r$ and $n+1-t=n+1-k+1$ or $j-1=s-1, n+1-k=n+1-t$ and $r=i+1$, or $i=r, n+1-k=n+1-t$ and $s-1=j-1+1$. This is equivalent to, 
    \begin{align*}
        j=s, i=r, \text{ and } k=t+1, \text{ or}\\
        j=s, k=t, \text{ and } r=i+1, \text{ or}\\
        i=r, k=t, \text{ and } s=j+1.
    \end{align*}
    Thus, $\{x_r,x_s,x_t\}=\{x_i,x_j,x_{k-1}\}$ or $\{x_r,x_s,x_t\}=\{x_{i+1},x_j,x_k\}$ or $\{x_r,x_s,x_t\}=\{x_i,x_{j+1},x_k\}$. Therefore, $\{x_i,x_j,x_k\}\{x_r,x_s,x_t\} \in E(\Gamma_3(P_n))$, so that $\psi$ is an isomorphism between $CS_n$ with the $3-$token graph $\Gamma_3(P_n)$.
\end{proof}

As a direct consequence of Theorem \ref{isomorfisGamma3Pn}, we obtain the following corollaries.

\begin{corollary}
    For any positive integer $n \geq 3$, the graph $\Gamma_3(P_n)$ has the following properties:
    \begin{enumerate}
        \item $\chi(\Gamma_3(P_n))=\omega(\Gamma_3(P_n))= \begin{cases}
            1, & \text{ if } n=3,\\
            2, & \text{ if } n > 3 \end{cases}$
        \item $\alpha(\Gamma_3(P_n)) = \begin{cases}
            \frac{n^3-3n^2+2n}{12}, & \text{ if } n \text{ is even},\\
            \frac{n^3-3n^2+5n-3}{12}, & \text{ if } n \text{ is odd}.
        \end{cases}$
        \item $diam(\Gamma_3(P_n))=3(n-3)$.
    \end{enumerate}
\end{corollary}
\begin{proof}
    It is clear from Theorem \ref{isomorfisGamma3Pn}, Theorem \ref{chiCS_n}, Theorem \ref{bilangan kromatik dan clique}, Theorem \ref{teo_alpha_CSn}, Theorem \ref{teo_diamCSn}. 
\end{proof}

To understand the structure of the automorphism group of \(\Gamma_3(P_n)\), where \(P_n\) represents the path graph with \(n \geq 3\), we delve into its symmetries and transformations. The automorphism group, \(Aut(\Gamma_3(P_n))\), captures all the graph's self-isomorphisms, preserving vertex connectivity. The following theorem establishes the precise characterization of \(Aut(\Gamma_3(P_n))\) based on the value of \(n\). 

\begin{theorem}\label{AutGamma3Pn}
   Let $P_n$, $n \geq 3$ be a path graph. Then,
    \begin{align*}
    Aut(\Gamma_3(P_n))\cong \begin{cases}
        \mathbb{Z}_1 &,\text{if} ~n=3,\\
        \mathbb{Z}_2 \times \mathbb{Z}_2 &,\text{if} ~n=6,\\
        \mathbb{Z}_2 &,\text{otherwise}.
    \end{cases} 
    \end{align*}
\end{theorem}
\begin{proof}
    For $n=3$, we know that the graph $\Gamma_3(P_n)$ is isomorphic to $K_1$. Then, $Aut(\Gamma_3(P_3))=\{id\} \cong \mathbb{Z}_1$.
    For $n \geq 4$ and $n \neq 6$, by Theorem \ref{isomorfisGamma3Pn}, we have $\Gamma_3(P_n) \cong CS_n$. Let $f$ be an isomorphism on $CS_n$. Then, $f$ is either an identity mapping $id$ or a bijection function $g$ that maps $\{i,j,k\}$ to $\{n-1-i,k,j\}$ for every $(i,j,k) \in V(CS_n)$. It is clear that $g^2=g$. Hence, $Aut(CS_n)=\{id,g\}$ which is isomorphic to the group $\mathbb{Z}_2$. Therefore, we get
    $$Aut(\Gamma_3(P_n)) \cong Aut(CS_n) \cong \mathbb{Z}_2.$$
    For $n = 6$, the graph $CS_6$ is given as follows. Let $f$ be an isomorphism on $CS_6$. Then, the possible functions for $f$ are only the identity function or the bijective functions $g_1, g_2$, or $g_3$, with the following mappings.
    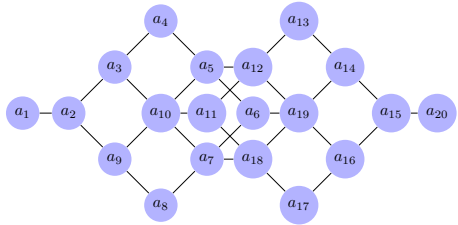
\begin{figure}[H]
		\begin{center} 
		\resizebox{0.5\textwidth}{!}{
                \begin{tikzpicture}  
				[scale=1,auto=center, roundnode/.style={circle, fill=blue!30}] 
				\node[roundnode] (a1) at (-6,0) {$a_1$};  
				\node[roundnode] (a2) at (-5,0)  {$a_2$}; 
				\node[roundnode] (a3) at (-4,1)  {$a_3$};
				\node[roundnode] (a4) at (-3,2) {$a_4$};
                \node[roundnode] (a5) at (-2,1)  {$a_5$};
				\node[roundnode] (a6) at (-1,0) {$a_6$};
				\node[roundnode] (a7) at (-2,-1) {$a_7$};
                \node[roundnode] (a8) at (-3,-2) {$a_8$};  
				\node[roundnode] (a9) at (-4,-1) {$a_9$}; 
				\node[roundnode] (a10) at (-3,0) {$a_{10}$};
				\node[roundnode] (a11) at (-2,-0) {$a_{11}$};
				\node[roundnode] (a12) at (-1,1) {$a_{12}$};
                \node[roundnode] (a13) at (0,2) {$a_{13}$};
				\node[roundnode] (a14) at (1,1) {$a_{14}$};
				\node[roundnode] (a15) at (2,0) {$a_{15}$};
                \node[roundnode] (a16) at (1,-1) {$a_{16}$}; 
				\node[roundnode] (a17) at (0,-2) {$a_{17}$};
				\node[roundnode] (a18) at (-1,-1) {$a_{18}$};
                \node[roundnode] (a19) at (0,0) {$a_{19}$};
				\node[roundnode] (a20) at (3,0) {$a_{20}$};
                
				\draw (a1) -- (a2);
                \draw (a2) -- (a3);
                \draw (a3) -- (a4);
                \draw (a4) -- (a5);
                \draw (a5) -- (a6);
                \draw (a6) -- (a7);
                \draw (a7) -- (a8);
                \draw (a8) -- (a9);
                \draw (a9) -- (a2);
                \draw (a10) -- (a3);
                \draw (a10) -- (a5);
                \draw (a10) -- (a7);
                \draw (a10) -- (a9);
                \draw (a10) -- (a11);
                \draw (a12) -- (a5);
                \draw (a18) -- (a7);
                \draw (a19) -- (a6);
                \draw (a11) -- (a12);
                \draw (a12) -- (a13);
                \draw (a13) -- (a14);
                \draw (a14) -- (a15);
                \draw (a15) -- (a16);
                \draw (a16) -- (a17);
                \draw (a17) -- (a18);
                \draw (a18) -- (a11);
                \draw (a19) -- (a12);
                \draw (a19) -- (a14);
                \draw (a19) -- (a16);
                \draw (a19) -- (a18);
                \draw (a15) -- (a20);
			\end{tikzpicture}}
		      \caption{Graph $CS_6$}
		\end{center}
	\end{figure}
    \begin{table}[H]
        \centering
        \begin{tabular}{|c|c|c|c|}
            $x \in V(CS_6)$ & $g_1(x)$ & $g_2(x)$ & $g_3(x)$\\ \hline
            $a_1$ & $a_1$ & $a_{20}$ & $a_{20}$ \\
            $a_2$ & $a_2$ & $a_{15}$ & $a_{15}$\\
            $a_3$ & $a_9$ & $a_{14}$ & $a_{16}$\\
            $a_4$ & $a_8$ & $a_{13}$ & $a_{17}$\\
            $a_5$ & $a_7$ & $a_{12}$ & $a_{18}$\\
            $a_6$ & $a_6$ & $a_{11}$ & $a_{11}$\\
            $a_7$ & $a_5$ & $a_{18}$ & $a_{12}$\\
            $a_8$ & $a_4$ & $a_{17}$ & $a_{13}$\\
            $a_9$ & $a_3$ & $a_{16}$ & $a_{14}$\\
            $a_{10}$ & $a_{10}$ & $a_{19}$ & $a_{19}$
        \end{tabular}\vspace{1cm}\begin{tabular}{|c|c|c|c|}
            $x \in V(CS_6)$ & $g_1(x)$ & $g_2(x)$ & $g_3(x)$\\ \hline
            $a_{11}$ & $a_{11}$ & $a_{6}$ & $a_{6}$\\
            $a_{12}$ & $a_{18}$ & $a_{5}$ & $a_{7}$\\
            $a_{13}$ & $a_{17}$ & $a_{4}$ & $a_{8}$\\
            $a_{14}$ & $a_{16}$ & $a_{3}$ & $a_{9}$\\
            $a_{15}$ & $a_{15}$ & $a_{2}$ & $a_{2}$\\
            $a_{16}$ & $a_{14}$ & $a_{9}$ & $a_{3}$\\
            $a_{17}$ & $a_{13}$ & $a_{4}$ & $a_{8}$\\
            $a_{18}$ & $a_{12}$ & $a_{7}$ & $a_{5}$\\
            $a_{19}$ & $a_{19}$ & $a_{10}$ & $a_{10}$\\
            $a_{20}$ & $a_{20}$ & $a_{1}$ & $a_{1}$
        \end{tabular}
        \vspace{-1cm}
        \caption{Automorphism of $CS_6$}
    \end{table}
    Thus, $Aut(CS_6)=\{id,g_1,g_2,g_3\}$, where $g_1^2=id$, $g_2^2=id$, and $g_3^2=id$, which is isomorphic to the group $\mathbb{Z}_2 \times \mathbb{Z}_2$.
\end{proof}

Observe the following independent edge set in the graph $\Gamma_3(P_n)$ for every $n\geq 4$:
\begin{align*} \small
    \left\{(i,j,k)(i,j,k+1): k=2t-1, 1 \leq i \leq n-2-k, 1 \leq j \leq i, 1 \leq t \leq \frac{n-2}{2} \right\} \cup \\
    \left\{(i,j,k)(i,j+1,k): i=2t, j=2s-1, k=n-1-2t, 1 \leq t \leq \frac{n-2}{2}, 1 \leq s \leq t \right \}
\end{align*}
for even $n$, and
\begin{align*} \small
    \left\{(i,j,k)(i,j,k+1): k=2t-1, 1 \leq i \leq n-2-k, 1 \leq j \leq i, 1 \leq t \leq \frac{n-3}{2}\right\} \cup \\
    \left\{(i,j,k)(i,j+1,k): i=2t, j=2s-1, k=n-2-2t, 1 \leq t \leq \frac{n-3}{2}, 1 \leq s \leq t\right\}
\end{align*}
for odd $n$. We hypothesize that these sets represent the largest independent edge sets that can be constructed. Based on this hypothesize, the following conjecture is proposed.
\begin{conjecture}
    For any $n\geq 3$, the $3$-token graph $\Gamma_3(P_n)$ satisfies the following.
    \begin{align*}
    \alpha'(\Gamma_3(P_n)) &= 
    \begin{cases}
        \sum_{k=1}^{\frac{n}{2}-1} \left(\sum_{i=1}^{2k-1} i \right)+ \sum_{i=1}^{\frac{n}{2}-1} i, & \text{ if } n \text{ is even},\\
            \sum_{k=1}^{\frac{n-1}{2}-1} \left(\sum_{i=1}^{2k} i \right)+ \sum_{i=1}^{\frac{n-1}{2}-1} i, & \text{ if } n \text{ is odd}
    \end{cases}\\
    &=\begin{cases}
            \frac{n^3-3n^2+2n}{12}, & \text{ if } n \text{ is even},\\
            \frac{n^3-3n^2-n+3}{12}, & \text{ if } n \text{ is odd}.
        \end{cases}
    \end{align*}
\end{conjecture}


\begin{openproblem} For further research, it would be interesting to investigate the structure of the 3-token graph of cycle graphs and other types of graphs. Additionally, the k-token graph of path graphs presents further open problems that could be explored.
\end{openproblem}






\section*{Acknowledgements}

This research work was supported by the 2025 FMIPA UGM Faculty Research Grant. The authors gratefully acknowledge the reviewers for their valuable feedback and constructive suggestions.



\end{document}